\documentclass[12pt,twoside]{amsart}
\usepackage{amsmath, amsthm, amscd, amsfonts, amssymb, graphicx}
\usepackage{enumerate}
\usepackage[colorlinks=true,
linkcolor=blue,
urlcolor=cyan,
citecolor=red]{hyperref}
\usepackage{mathrsfs}
\newtheorem{theorem}{Theorem}[section]
\newtheorem{lemma}{Lemma}[section]
\newtheorem{remark}{Remark}[section]

\newtheorem{corollary}{Corollary}[section]

\newtheorem{proposition}{Proposition}[section]
\numberwithin{equation}{section}

\begin{document}
	
\title{A convex-block approach for  numerical radius inequalities}
\author{Mohammad Sababheh, Cristian Conde, and Hamid Reza Moradi}
\subjclass[2010]{Primary 47A12, 47A30, Secondary 15A60, 47B15}
\keywords{Numerical radius, norm inequality, Cartesian decomposition, triangle inequality}

\begin{abstract}
This article implements a simple convex approach and block techniques to obtain several new refined versions of numerical radius inequalities for Hilbert space operators. This includes comparisons among the norms of the operators, their Cartesian parts, their numerical radii, the numerical radius of the product of two operators, and the Aluthge transform.
\end{abstract}
\maketitle
\pagestyle{myheadings}
\markboth{\centerline {}}
{\centerline {}}
\bigskip
\bigskip

\section{Introduction}
Let $\mathcal{B}(\mathcal{H})$ denote the $C^*-$algebra of all bounded linear operators acting on a Hilbert space $\mathcal{H}$. An operator $T\in\mathcal{B}(\mathcal{H})$ is said to be positive if, for every $x\in\mathcal{H}$, one has $\left<Tx,x\right>\geq 0$. In this case, we simply write $A\geq O.$ Positive operators play an important role in understanding the geometry of a Hilbert space, and these operators constitute a special class of the wider class of self-adjoint operators; that is $A^*=A$, where $A^*$ denotes the conjugate of $A$. Among the most basic properties of self-adjoint operators is the fact that 
$$\|T\|=\omega(T)=r(T),\;T\;{\text{is\;normal}},$$ where $\|\cdot\|, \omega(\cdot)$, and $r(\cdot)$ denote the operator norm, the numerical radius, and the spectral radius respectively. Actually, for a general $T\in\mathcal{B}(\mathcal{H})$ one has 
 $$\|T\|\geq \omega(T)\geq r(T).$$\\
While both $\|\cdot\|$ and $\omega(\cdot)$ are  norms on $\mathcal{B}(\mathcal{H})$, $r(\cdot)$ is not. In fact, we have the equivalence relation \cite[Theorem 1.3-1]{gust}
\begin{equation}\label{eq_equiv_norms}
\frac{1}{2}\|T\|\le\omega(T)\le\|T\|,\;T\in\mathcal{B}(\mathcal{H}).
\end{equation}

Numerous researchers' core interests have been sharpening the above inequality and obtaining new possible relations between $\|\cdot\|$ and $\omega(\cdot)$. This is because $\|\cdot\|$ is much easier to compute than $\omega(\cdot)$, not to forget the math appetite for obtaining such new relations.

The Cartesian decomposition of $T\in\mathcal{B}(\mathcal{H})$ is $T=\mathfrak{R}T+\textup i\mathfrak{I}T$, where  $\mathfrak{R}T=\frac{T+T^*}{2}$ and $\mathfrak{I}T=\frac{T-T^*}{2\textup i}$ are the real and imaginary parts of $T$, respectively. Although $\|T\|\geq \omega(T)$ is always valid, the following reverses hold for the Cartesian components of $T$, see \cite[Theorem 2.1]{1}
\begin{equation}\label{eq_norm_reim}
\|\mathfrak{R}T\|\le\omega(T),\;\|\mathfrak{I}T\|\le \omega(T).
\end{equation}

While the original definition of $\omega(\cdot)$ is based on a  supremum over inner product values (i.e., $\omega(T)=\sup_{\|x\|=1}|\left<Tx,x\right>|$), the following identity is extremely useful \cite{3}
\begin{equation}\label{eq_w_re}
\underset{\theta \in \mathbb{R}}{\mathop{\sup }}\,\left\| {{\operatorname{\mathfrak Re}}^{\textup i\theta }}T \right\|=\omega \left( T \right).
\end{equation}

Exploring further relations between $\|\cdot\|$ and $\omega(\cdot),$ 
it has been shown in  \cite[Theorem 2.3]{1} that
\begin{equation}\label{eq_fuad_mos}
\|A+B\|\le 2\omega\left(\left[\begin{array}{cc}O&A\\B^*&O\end{array}\right]\right)\le \|A\|+\|B\|;
\end{equation}
as an interesting refinement of the triangle inequality of norms, using the numerical radius of a matrix operator.

Having the matrix operator term in \eqref{eq_fuad_mos} is not a coincidence. In fact, numerous results have included such terms while studying numerical radius inequalities. For example,  it has been shown in \cite[Theorem 2.4]{4} that
\begin{equation}\label{eq_max_w}
\frac{\max \left\{ \omega \left( S+T \right),\omega \left( S-T \right) \right\}}{2}\le \omega \left( \left[ \begin{matrix}
   O & S  \\
   T & O  \\
\end{matrix} \right] \right),\text{ for any }S,T\in \mathcal B\left( \mathcal H \right);
\end{equation}

an inequality which has been reversed in a way or another by the form \cite[Theorem 2.4]{4}
\begin{equation}\label{eq_w_average_w}
\omega \left( \left[ \begin{matrix}
   O & S  \\
   T & O  \\
\end{matrix} \right] \right)\le \frac{\omega \left( S+T \right)+\omega \left( S-T \right)}{2},\text{ for any }S,T\in \mathcal B\left( \mathcal H \right).
\end{equation}

The above matrix operator is not only comparable with numerical radius terms, as we also have \cite[Theorem 2.1]{5} 
\begin{equation}\label{eq_need_prf}
2\omega \left( \left[ \begin{matrix}
   O & A  \\
   {{B}^{*}} & O  \\
\end{matrix} \right] \right)\le \max \left\{ \left\| A \right\|,\left\| B \right\| \right\}+\frac{1}{2}\left( \left\| {{\left| A \right|}^{\frac{1}{2}}}{{\left| B \right|}^{\frac{1}{2}}} \right\|+\left\| {{\left| {{B}^{*}} \right|}^{\frac{1}{2}}}{{\left| {{A}^{*}} \right|}^{\frac{1}{2}}} \right\| \right),
\end{equation}
for any $A,B\in \mathcal B\left( \mathcal H \right)$.

The right-hand side of this latter inequality is related to the Davidson-Power inequality, which has been generalized in \cite[Theorem 5]{6} to the form
\begin{equation}\label{eq_abuamer}
\|A+B^*\|\le \max\{\|A\|,\|B\|\}+\max\{\|\;|A|^{1/2}|B^*|^{1/2}\|,\|\;|A^*|^{1/2}|B|^{1/2}\|\}.
\end{equation}

An important tool in obtaining matrix inequalities is convexity; whether it is scalar or operator convexity. Recall that a function $f:J\to\mathbb{R}$ is said to be convex on the interval $J$ if it satisfies $f((1-\lambda)a+\lambda b)\le (1-\lambda)f(a)+\lambda f(b)$ for all $a,b\in J$ and $0\le \lambda\le 1$. In convex analysis, the Hermite-Hadamard inequality which states that for a convex function $f$ on $[0,1]$ one has
\begin{equation}\label{eq_hh}
f\left( \frac{1}{2} \right)\le \int\limits_{0}^{1}{f\left( t \right)dt}\le \frac{f\left( 0 \right)+f\left( 1 \right)}{2},
\end{equation}
is a non-avoidable tool. Notice that this inequality provides a refinement of the mid-convexity condition of $f$.

Our target in this paper is to further explore numerical radius and operator norm inequalities, via matrix  operators and convex functions. For this, we begin by noting that 
since $\omega(\cdot)$ and $||\cdot||$ are norms, one can  easily verify that the functions
\[f\left( t \right)=\omega \left( \left( 1-t \right)T+t{{T}^{*}} \right),\;{\text{and}}\;g(t)=\left\| \left( 1-t \right)T+t{{T}^{*}} \right\|\]
are convex on $\left[ 0,1 \right]$.

With a considerable amount of research devoted to inequalities of convex functions, the following inequalities which have been shown in \cite{2}  for a convex function $f:\left[ 0,1 \right]\to \mathbb{R}$ have played a useful role in the literature
	\[f\left( t \right)\le \left( 1-t \right)f\left( 0 \right)+tf\left( 1 \right)-2r\left( \frac{f\left( 0 \right)+f\left( 1 \right)}{2}-f\left( \frac{1}{2} \right) \right),\]
and
\[\left( 1-t \right)f\left( 0 \right)+tf\left( 1 \right)\le f\left( t \right)+2R\left( \frac{f\left( 0 \right)+f\left( 1 \right)}{2}-f\left( \frac{1}{2} \right) \right),\]
where $r=\min \left\{ t,1-t \right\}$, $R=\max \left\{ t,1-t \right\}$, and $0\le t\le 1$. 
We refer the reader to \cite{sab_mia,sab_mjom} for some applications and further discussion of these inequalities.

Applying these later inequalities to the convex functions $f$ and $g$ above implies the following refinements and reverses of \eqref{eq_norm_reim}.
\begin{equation}\label{6}
\frac{\omega \left( T \right)-\omega \left( \left( 1-t \right)T+t{{T}^{*}} \right)}{2R}\le \omega \left( T \right)-\left\| \mathfrak RT \right\|\le \frac{\omega \left( T \right)-\omega \left( \left( 1-t \right)T+t{{T}^{*}} \right)}{2r}.
\end{equation}
Furthermore,
\begin{equation}\label{ned_quo}
\frac{\left\| T \right\|-\left\| \left( 1-t \right)T+t{{T}^{*}} \right\|}{2R}\le \left\| T \right\|-\left\| \mathfrak RT \right\|\le \frac{\left\| T \right\|-\left\| \left( 1-t \right)T+t{{T}^{*}} \right\|}{2r}.
\end{equation}

Using this approach, we will be able to present refined versions and generalizations of most of the above inequalities, with the conclusion of some product inequalities that entail some interesting relations. We refer to inequalities that govern $\omega(AB)$ as product inequalities. It is well-known that $\omega(\cdot)$ is not sub-multiplicative. We refer the reader to \cite{Li} for further discussion of this property. Interestingly, our approach will entail a relation between $\omega(AB)$ and $\|A+B\|$, with an application to the matrix arithmetic-geometric mean inequality that states \cite[Theorem IX.4.5]{bhatia}
\begin{equation*}
\|A^{1/2}B^{1/2}\|\le\frac{1}{2}\|A+B\|,\;A,B\in\mathcal{B}(\mathcal{H}), A,B\geq O. 
\end{equation*}
Namely, we obtain a new refinement of this inequality using the numerical radius; as a new approach to this direction, see Remark \ref{remark_amgm} below.

To achieve our goal, some auxiliary results are needed as follows. 
\begin{lemma}
Let $A,B\in\mathcal{B}(\mathcal{H})$. 
\begin{enumerate}
\item If $n\in\mathbb{N}$, then \cite[Theorem 2.1-1]{gust}
\begin{equation}\label{eq_power_ineq}
\omega(A^n)\le \omega^n(A).
\end{equation}
\item The operator norm satisfies the identity
\begin{equation}\label{eq_norm_blocks}
\left\|\left[\begin{array}{cc}O&A\\A^*&O\end{array}\right]\right\|=\|A\|.
\end{equation}
\end{enumerate}
\end{lemma}

\section{Main Result}
In this section we present our results, starting with the following simple consequence that follows by applying \eqref{eq_hh} on
\[f\left( t \right)=\omega \left( \left( 1-t \right)T+t{{T}^{*}} \right)\;{\text{and}}\;g(t)=\left\| \left( 1-t \right)T+t{{T}^{*}} \right\|\] yielding refinements of \eqref{eq_norm_reim}.
\begin{proposition}
Let $T\in \mathcal B\left( \mathcal H \right)$. Then
\begin{equation}\label{2}
\left\| \mathfrak RT \right\|\le \int\limits_{0}^{1}{\omega \left( \left( 1-t \right)T+t{{T}^{*}} \right)}dt\le \omega \left( T \right),
\end{equation}
and
\begin{equation}\label{3}
\left\| \mathfrak IT \right\|\le \int\limits_{0}^{1}{\omega \left( \left( 1-t \right){{T}^{*}}-tT \right)dt}\le \omega \left( T \right).
\end{equation}
Moreover,
\begin{equation}\label{15}
\left\| \mathfrak RT \right\|\le \int\limits_{0}^{1}{\left\| \left( 1-t \right)T+t{{T}^{*}} \right\|dt}\le \left\| T \right\|,
\end{equation}
and
\begin{equation}\label{16}
\left\| \mathfrak IT \right\|\le \int\limits_{0}^{1}{\left\| \left( 1-t \right){{T}^{*}}-tT \right\|dt}\le \left\| T \right\|.
\end{equation}
\end{proposition}

The identity \eqref{eq_w_re} provides an alternative formula to evaluate the numerical radius without appealing to the inner product. Interestingly, the inequalities \eqref{2} and \eqref{3} provide the following alternative identities, which help better understand how the numerical radius behaves.
\begin{corollary}
Let $T\in \mathcal B\left( \mathcal H \right)$. Then
\[\omega \left( T \right)=\underset{\theta \in \mathbb{R}}{\mathop{\sup }}\,\left(\int\limits_{0}^{1}{\omega \left( \left( 1-t \right){{e}^{\textup i\theta }}T+t{{e}^{-\textup i\theta }}{{T}^{*}} \right)dt}\right)=\underset{\theta \in \mathbb{R}}{\mathop{\sup }}\,\left(\int\limits_{0}^{1}{\omega \left( \left( 1-t \right){{e}^{-\textup i\theta }}{{T}^{*}}-t{{e}^{\textup i\theta }}T \right)dt}\right).\]
\end{corollary}
\begin{proof}
Replacing $T$ by ${{e}^{\textup i\theta }}T$  in \eqref{2}, we get
\[\left\| {{\operatorname{\mathfrak Re}}^{\textup i\theta }}T \right\|\le \int\limits_{0}^{1}{\omega \left( \left( 1-t \right){{e}^{\textup i\theta }}T+t{{e}^{-\textup i\theta }}{{T}^{*}} \right)dt}\le \omega \left( T \right).\]
Taking the supremum over $\theta \in \mathbb{R}$,  \eqref{eq_w_re} implies the first identity. The second identity follows from \eqref{3} and noting that
\[\underset{\theta \in \mathbb{R}}{\mathop{\sup }}\,\left\| \mathfrak I{{e}^{\textup i\theta }}T \right\|=\omega \left( T \right).\]
\end{proof}

The following result involves an integral refinement of the second inequality in \eqref{eq_equiv_norms}.
\begin{proposition}
Let $T\in \mathcal B\left( \mathcal H \right)$. Then
\[\omega \left( T \right)\le \min \left\{ {{\lambda }_{1}},{{\lambda }_{2}} \right\}\le \left\| T \right\|,\]
where 
\[{{\lambda }_{1}}=\underset{\theta \in \mathbb{R}}{\mathop{\sup }}\,\left( \int\limits_{0}^{1}{\left\| \left( 1-t \right){{e}^{\textup i\theta }}T+t{{e}^{-\textup i\theta }}{{T}^{*}} \right\|dt} \right)\text{ and }{{\lambda }_{2}}=\underset{\theta \in \mathbb{R}}{\mathop{\sup }}\,\left( \int\limits_{0}^{1}{\left\| \left( 1-t \right){{e}^{-\textup i\theta }}{{T}^{*}}-t{{e}^{\textup i\theta }}T \right\|dt} \right).\]
\end{proposition}
\begin{proof}
By the inequality \eqref{15}, we have 
\begin{equation*}
\sup_{\theta \in \mathbb{R}}\|{{\operatorname{\mathfrak Re}}^{\textup i\theta }}T\|\le \sup_{\theta \in \mathbb{R}}\left( \int\limits_{0}^{1}{\left\| \left( 1-t \right)e^{i\theta}T+t{e^{-i\theta}{T}^{*}} \right\|dt}\right)\le \|T\|.
\end{equation*}
Finally, by \eqref{eq_w_re} we get
\begin{equation*}
 \omega(T)\le\sup_{\theta \in \mathbb{R}}\left( \int\limits_{0}^{1}{\left\| \left( 1-t \right)e^{i\theta}T+t{e^{-i\theta}{T}^{*}} \right\|dt}\right)\le \left\| T \right\|.
\end{equation*}
By a similar proof and with the help of \eqref{16}, we also have
\begin{equation*}
\omega(T)\le \sup_{\theta\in \mathbb{R}}\left( \int\limits_{0}^{1}{\left\| \left( 1-t \right){e^{-i\theta}{T}^{*}}-te^{i\theta}T \right\|dt}\right)\le \left\| T \right\|.
\end{equation*}
This completes the proof.
\end{proof}

The second inequality in the inequalities \eqref{2} and \eqref{3} can be reversed as follows. 
\begin{proposition}
Let $T\in \mathcal B\left( \mathcal H \right)$. Then 
\[\frac{1}{2}\omega \left( T \right)\le \left\{ \begin{aligned}
  & \int\limits_{0}^{1}{\omega \left( \left( 1-t \right)T+t{{T}^{*}} \right)dt}, \\ 
 & \int\limits_{0}^{1}{\omega \left( \left( 1-t \right){{T}^{*}}-tT \right)dt}. \\ 
\end{aligned} \right.\]
\end{proposition}
\begin{proof}
For any $0\le t\le 1$, it can be easily shown that
\[\left| 1-2t \right|\omega \left( T \right)\le \min \left\{ \omega \left( \left( 1-t \right)T+t{{T}^{*}} \right),\omega \left( \left( 1-t \right){{T}^{*}}-tT \right) \right\}.\]
Integrating this over the interval $[0,1]$ implies the desired result.
\end{proof}

The following result holds as well.
\begin{theorem}
	Let $T\in \mathcal B\left( \mathcal H \right)$. Then
\begin{equation*}
\left\| T \right\|\le  2\int\limits_{0}^{1}{\left\| \left( 1-t \right)T+t{{T}^{*}} \right\|dt}\le2 \left\| T \right\|,
\end{equation*}
and
\[\omega \left( T \right)\le 2\int\limits_{0}^{1}{\omega \left( \left( 1-t \right)T+t{{T}^{*}} \right)dt}\le2\omega \left( T \right).\]
\end{theorem}
\begin{proof}
Let $h(T)=\|T\|$ for any $T\in \mathcal B\left( \mathcal H \right)$. Then, $h$ is a convex function on $\mathcal B\left( \mathcal  H \right)$. For each $t\in [0, 1],$ we have 
\begin{equation*}
h((1-2t)T)+h((2t-1)T^*)=h((1-t)A+tB)+h((1-t)B+tA),
\end{equation*}
where $A=(1-t)T+tT^*$ and $B=-(1-t)T^*-tT$. Then, 
\[ \begin{aligned}
h((1-2t)T)+h((2t-1)T^*)&=h((1-t)A+tB)+h((1-t)B+tA)\\
&\le(1-t)h(A)+th(B)+(1-t)h(B)+th(A)\\
&=h(A)+h(B)\nonumber \\
&=h((1-t)T+tT^*)+h(-(1-t)T^*-tT)\\
&=h((1-t)T+tT^*)+h((1-t)T^*+tT).
\end{aligned} \]
Integrating, the previous inequality, from $t=0$ to $t=1$, we obtain
\begin{equation*}
\int_0^1|1-2t|(\|T\|+\|T^*\|)\:dt\le 2\int_0^1\|(1-t)T+tT^*\|\:dt.
\end{equation*}
Thus, 
\begin{equation*}
\|T\|=\frac12 (\|T\|+\|T^*\|)\le 2\int_0^1\|(1-t)T+tT^*\|\:dt.
\end{equation*}
On the other hand,
\begin{align*}
\|(1-t)T+tT^*\|&\le (1-t)\|T\|+t\|T^*\|=\|T\|; 0\le t\le 1.
\end{align*}
Integrating this last inequality and then multiplying by 2 complete the proof of the first inequality. The second inequality is proved similarly.
\end{proof}

Continuing with the convexity of the norms, the inequality \eqref{6} may be used to get the following refinement of the first inequality in \eqref{eq_equiv_norms}.
\begin{theorem}\label{9}
Let $T\in \mathcal B\left( \mathcal H \right)$. Then for any $0\le t\le 1$,
\[\frac{1}{2}\left\| T \right\|+\frac{1}{4R}\left( 2\omega \left( T \right)-\left( \omega \left( \left( 1-t \right){{T}^{*}}-tT \right)+\omega \left( \left( 1-t \right)T+t{{T}^{*}} \right) \right) \right)\le \omega \left( T \right),\]
where $R=\max \left\{ t,1-t \right\}$.
\end{theorem}
\begin{proof}
The first inequality in \eqref{6} can be written as
\begin{equation}\label{4}
\left\| \mathfrak RT \right\|+\frac{\omega \left( T \right)-\omega \left( \left( 1-t \right)T+t{{T}^{*}} \right)}{2R}\le \omega \left( T \right).
\end{equation}
Replacing $\textup i{{T}^{*}}$ by $T$, we infer that
\begin{equation}\label{5}
\left\| \mathfrak IT \right\|+\frac{\omega \left( T \right)-\omega \left( \left( 1-t \right){{T}^{*}}-tT \right)}{2R}\le \omega \left( T \right).
\end{equation}
By \eqref{4} and \eqref{5}, we get
\[\begin{aligned}
  & \frac{1}{2}\left\| T \right\|+\frac{1}{4R}\left( 2\omega \left( T \right)-\left( \omega \left( \left( 1-t \right){{T}^{*}}-tT \right)+\omega \left( \left( 1-t \right)T+t{{T}^{*}} \right) \right) \right) \\ 
 & =\frac{1}{2}\left\| \mathfrak RT+\textup i\mathfrak IT \right\|+\frac{1}{4R}\left( 2\omega \left( T \right)-\left( \omega \left( \left( 1-t \right){{T}^{*}}-tT \right)+\omega \left( \left( 1-t \right)T+t{{T}^{*}} \right) \right) \right) \\ 
 & \le \frac{1}{2}\left( \left\| \mathfrak RT \right\|+\left\| \mathfrak IT \right\| \right)+\frac{1}{4R}\left( 2\omega \left( T \right)-\left( \omega \left( \left( 1-t \right){{T}^{*}}-tT \right)+\omega \left( \left( 1-t \right)T+t{{T}^{*}} \right) \right) \right) \\ 
 &\qquad\text{(by the triangle inequality for the usual operator norm)}\\
 & \le \omega \left( T \right).
\end{aligned}\]
This completes the proof.
\end{proof}

As a consequence of Theorem \ref{9}, we get the following corollaries. Our results considerably refines \cite[(4.3)]{4} and \cite[(4.2)]{4}, respectively.
\begin{corollary}\label{refwn}
	Let $T\in \mathcal B\left( \mathcal H \right)$. Then,
\begin{equation*}
\frac{1}{2}\left\| T \right\|+\frac{1}{2}\Big| \|\mathfrak IT \|-\|\mathfrak RT \|\Big| \le
\frac{1}{2}\left\| T \right\|+\frac{1}{2}\left( 2\omega \left( T \right)-\left( \|\mathfrak IT \|+\|\mathfrak RT \| \right)\right)\le \omega \left( T \right).	
\end{equation*}
\end{corollary}
\begin{proof}
The second inequality can be deduced from Theorem \ref{9} with $t=\frac12.$ On the other hand, we have 
\[\begin{aligned}
\frac{1}{2}\Big| \|\mathfrak IT \|-\|\mathfrak RT \|\Big|=& \frac{1}{2}\Big| \|\mathfrak IT \|-\omega(T)+\omega(T)-|\mathfrak RT \|\Big| \\ 
&\le\frac{1}{2}\left(\Big| \|\mathfrak IT \|-\omega(T)\Big|+\Big|\omega(T)-|\mathfrak RT \|\Big|\right) \\
&= \frac{1}{2}\left( \omega(T)-\|\mathfrak IT \|+\omega(T)-\|\mathfrak RT \|\right),\quad({\text{by\;the\;inequality\;\eqref{eq_w_re}}}). \nonumber\ 
\end{aligned}\]	
This completes the proof.
\end{proof}

As a consequence of Corollary \ref{refwn}, we characterize when the numerical radius to be equal to half the operator norm. The following result is related to Theorem 3.1 previously obtained by Yamazaki in \cite{3}.
\begin{proposition}
	Let $T\in \mathcal B\left( \mathcal H \right)$. Then, $\frac{\|T\|}{2}=\omega(T)$ if and only if $\|\mathfrak Ie^{\textup i\theta}T\|=\|\mathfrak Re^{\textup i\theta}T\|=\frac{\|T\|}{2}$ for any $\theta \in \mathbb{R}.$
\end{proposition}
\begin{proof}
	If $\|\mathfrak Ie^{\textup i\theta}T\|=\|\mathfrak Re^{\textup i\theta}T\|=\frac{\|T\|}{2}$ for any $\theta \in \mathbb{R}$, then by \eqref{eq_w_re} we conclude that $\omega(T)=\frac{\|T\|}{2}$. Conversely, we suppose that $\omega(T)=\frac{\|T\|}{2}$, thus from Corollary \ref{refwn} we conclude that 
\[	\frac{1}{2}\left\| T \right\|=\frac{1}{2}\left\| T \right\|+\frac{1}{2}\Big| \|\mathfrak IT \|-\|\mathfrak RT \|\Big|=
	\frac{1}{2}\left\| T \right\|+\frac{1}{2}\left( 2\omega \left( T \right)-\left( \|\mathfrak IT \|+\|\mathfrak RT \| \right)\right)= \omega \left( T \right).\]
	If we replace $T$ for $e^{\textup i\theta}T$ with $\theta \in \mathbb{R}$, we have 
	\[\frac{1}{2}\left\| T \right\|=	\frac{1}{2}\left\| T \right\|+\frac{1}{2}\Big| \|\mathfrak Ie^{\textup i\theta}T \|-\|\mathfrak Re^{\textup i\theta}T \|\Big|=
	\frac{1}{2}\left\| T \right\|+\frac{1}{2}\left( 2\omega \left( T \right)-\left( \|\mathfrak Ie^{\textup i\theta}T \|+\|\mathfrak Re^{\textup i\theta}T \| \right)\right)= \omega \left( T \right).\]
	This implies that $\|\mathfrak Ie^{\textup i\theta}T \|=\|\mathfrak Re^{\textup i\theta}T \|$ and $2\omega(T)=\|\mathfrak Ie^{\textup i\theta}T \|+\|\mathfrak Re^{\textup i\theta}T \|,$ i.e. for any $\theta \in \mathbb{R}$ we get
	$$\|\mathfrak Ie^{\textup i\theta}T \|=\|\mathfrak Re^{\textup i\theta}T \|=\frac{\|T\|}{2}.$$
	\end{proof}

\begin{corollary}
Let $A,B\in \mathcal B\left( \mathcal H \right)$. Then, 
\begin{eqnarray}
\omega \left( \left[ \begin{matrix}
O & A  \\
{{B}} & O  \\
\end{matrix} \right]\right)&\geq&
\frac{1}{2}\left\| \left[ \begin{matrix}
O & A  \\
{{B}} & O  \\
\end{matrix} \right] \right\|+\frac{1}{2}\left( 2\omega \left( \left[ \begin{matrix}
O & A  \\
{{B}} & O  \\
\end{matrix} \right] \right)-\left( \|A-B^* \|+\| A+B^* \| \right)\right)\nonumber\\
&\geq&\frac{1}{2}\left\| \left[ \begin{matrix}
O & A  \\
{{B}} & O  \\
\end{matrix} \right] \right\|+\frac{1}{2}\Big| \|A-B^*\|-\|A+B^*\|\Big|.\nonumber\\ 
\nonumber\ \end{eqnarray}
\end{corollary}
\begin{proof}
	This follows clearly from Corollary \ref{refwn} by considering $T=\left[ \begin{matrix}
	O & A  \\
	{{B}} & O  \\
	\end{matrix} \right] $ and equality \eqref{eq_norm_blocks}.
\end{proof}

On the other hand, the reverse for the second inequality in \eqref{eq_w_re} may be obtained as follows. 
\begin{theorem}
Let $T\in \mathcal B\left( \mathcal H \right)$. Then for any $0\le t\le 1$,
\[\left\| T \right\|\le \omega \left( T \right)+\frac{\left\| T \right\|-\left\| \left( 1-t \right)T+t{{T}^{*}} \right\|}{2r}-\frac{\omega \left( T \right)-\omega \left( \left( 1-t \right)T+t{{T}^{*}} \right)}{2R},\]
where $r=\min \left\{ t,1-t \right\}$ and $R=\max \left\{ t,1-t \right\}$. In particular, 
\[\frac{\omega \left( T \right)-\omega \left( \left( 1-t \right)T+t{{T}^{*}} \right)}{2R}\le \frac{\left\| T \right\|-\left\| \left( 1-t \right)T+t{{T}^{*}} \right\|}{2r}.\]

\end{theorem}
\begin{proof}
The inequalities \eqref{6} and \eqref{ned_quo} imply
	\[\begin{aligned}
   \left\| T \right\|&\le \left\| \mathfrak RT \right\|+\frac{\left\| T \right\|-\left\| \left( 1-t \right)T+t{{T}^{*}} \right\|}{2r} \\ 
 & \le \omega \left( T \right)+\frac{\left\| T \right\|-\left\| \left( 1-t \right)T+t{{T}^{*}} \right\|}{2r}-\frac{\omega \left( T \right)-\omega \left( \left( 1-t \right)T+t{{T}^{*}} \right)}{2R}.  
\end{aligned}\]	
This proves the first assertion. The second assertion follows from the first, noting that $\omega(T)\le\|T\|.$
\end{proof}

Continuing with the theme of this paper, in the following result, the numerical radius of convex combinations of operator matrices is used to refine the triangle inequality, thanks to
\[\omega \left( \left[ \begin{matrix}
   O & \left( 1-t \right)A+tB  \\
   \left( 1-t \right){{B}^{*}}+t{{A}^{*}} & O  \\
\end{matrix} \right] \right)\le \omega \left( \left[ \begin{matrix}
   O & A  \\
   {{B}^{*}} & O  \\
\end{matrix} \right] \right); 0\le t\le 1.\]
\begin{theorem}\label{10}
Let $A,B\in \mathcal B\left( \mathcal H \right)$. Then for any $0\le t\le 1$,
\[\left\| A+B \right\|\le \left\| A \right\|+\left\| B \right\|-\frac{\omega \left( \left[ \begin{matrix}
   O & A  \\
   {{B}^{*}} & O  \\
\end{matrix} \right] \right)-\omega \left( \left[ \begin{matrix}
   O & \left( 1-t \right)A+tB  \\
   \left( 1-t \right){{B}^{*}}+t{{A}^{*}} & O  \\
\end{matrix} \right] \right)}{R},\]
where $R=\max \left\{ t,1-t \right\}$. 
\end{theorem}
\begin{proof}
Let $T=\left[ \begin{matrix}
   O & A  \\
   {{B}^{*}} & O  \\
\end{matrix} \right]$ on $\mathcal H\oplus \mathcal H$. Then by \eqref{6}, we can write  
{\footnotesize
\[\begin{aligned}
  & \left\| A+B \right\| \\ 
 & =\left\| T+{{T}^{*}} \right\| \\ 
 & =2\left\| \mathfrak RT \right\| \\ 
 & \le 2\omega \left( T \right)-\frac{\omega \left( T \right)-\omega \left( \left( 1-t \right)T+t{{T}^{*}} \right)}{R} \\ 
 & =2\underset{\theta \in \mathbb{R}}{\mathop{\sup }}\,\left\| \mathfrak R{{e}^{\textup i\theta }}T \right\|-\frac{\omega \left( T \right)-\omega \left( \left( 1-t \right)T+t{{T}^{*}} \right)}{R} \\ 
 & =\underset{\theta \in \mathbb{R}}{\mathop{\sup }}\,\left\| \left[ \begin{matrix}
   O & {{e}^{\textup i\theta }}A+{{e}^{-\textup i\theta }}B  \\
   {{e}^{\textup i\theta }}{{B}^{*}}+{{e}^{-\textup i\theta }}{{A}^{*}} & O  \\
\end{matrix} \right] \right\|-\frac{\omega \left( \left[ \begin{matrix}
   O & A  \\
   {{B}^{*}} & O  \\
\end{matrix} \right] \right)-\omega \left( \left[ \begin{matrix}
   O & \left( 1-t \right)A+tB  \\
   \left( 1-t \right){{B}^{*}}+t{{A}^{*}} & O  \\
\end{matrix} \right] \right)}{R} \\ 
 & =\underset{\theta \in \mathbb{R}}{\mathop{\sup }}\,\left\| {{e}^{\textup i\theta }}A+{{e}^{-\textup i\theta }}B \right\|-\frac{\omega \left( \left[ \begin{matrix}
   O & A  \\
   {{B}^{*}} & O  \\
\end{matrix} \right] \right)-\omega \left( \left[ \begin{matrix}
   O & \left( 1-t \right)A+tB  \\
   \left( 1-t \right){{B}^{*}}+t{{A}^{*}} & O  \\
\end{matrix} \right] \right)}{R} \\ 
 & \le \left\| A \right\|+\left\| B \right\|-\frac{\omega \left( \left[ \begin{matrix}
   O & A  \\
   {{B}^{*}} & O  \\
\end{matrix} \right] \right)-\omega \left( \left[ \begin{matrix}
   O & \left( 1-t \right)A+tB  \\
   \left( 1-t \right){{B}^{*}}+t{{A}^{*}} & O  \\
\end{matrix} \right] \right)}{R}, 
\end{aligned}\]}
where the triangle inequality for the operator norm has been used to obtain the last inequality. This completes the proof.
\end{proof}

\begin{remark}
Letting $T=\left[\begin{array}{cc}O&A\\B^*&O\end{array}\right]$, we have
\[\begin{aligned}
  & \omega \left( \left[ \begin{matrix}
   O & \left( 1-t \right)A+tB  \\
   \left( 1-t \right){{B}^{*}}+t{{A}^{*}} & O  \\
\end{matrix} \right] \right)+\omega \left( \left[ \begin{matrix}
   O & \left( 1-t \right)B-tA  \\
   \left( 1-t \right){{A}^{*}}-t{{B}^{*}} & O  \\
\end{matrix} \right] \right) \\ 
&=\omega((1-t)T+tT^*)+\omega((1-t)T^*-tT)\\
&\le 2\omega(T)\quad({\text{by\;the\;triangle\;inequality}})\\
 & = 2\omega \left( \left[ \begin{matrix}
   O & A  \\
   {{B}^{*}} & O  \\
\end{matrix} \right] \right),
\end{aligned}\]
for any $0\le t\le 1$. 
Thus, noting \eqref{eq_max_w} we have
\[\begin{aligned}
  & \max \left\{ \omega \left( \left( 1-t \right)\left( B+{{A}^{*}} \right)-t\left( A+{{B}^{*}} \right) \right),\omega \left( \left( 1-t \right)\left( B-{{A}^{*}} \right)+t\left( {{B}^{*}}-A \right) \right) \right\} \\ 
 &\quad +\max \left\{ \omega \left( \left( 1-t \right)\left( A+{{B}^{*}} \right)+t\left( B+{{A}^{*}} \right) \right),\omega \left( \left( 1-t \right)\left( A-{{B}^{*}} \right)+t\left( B-{{A}^{*}} \right) \right) \right\} \\ 
 & \le 2\omega \left( \left[ \begin{matrix}
   O & \left( 1-t \right)A+tB  \\
   \left( 1-t \right){{B}^{*}}+t{{A}^{*}} & O  \\
\end{matrix} \right] \right)+2\omega \left( \left[ \begin{matrix}
   O & \left( 1-t \right)B-tA  \\
   \left( 1-t \right){{A}^{*}}-t{{B}^{*}} & O  \\
\end{matrix} \right] \right) \\ 
 & \le 4\omega \left( \left[ \begin{matrix}
   O & A  \\
   {{B}^{*}} & O  \\
\end{matrix} \right] \right).
\end{aligned}\]
 In particular,
 \begin{equation}\label{8}
\begin{aligned}
   & \max \left\{ \left\| \mathfrak IA-\mathfrak IB \right\|,\left\| \mathfrak RA-\mathfrak RB \right\| \right\}+\max \left\{ \left\| \mathfrak RA+\mathfrak RB \right\|,\left\| \mathfrak IA+\mathfrak IB \right\| \right\} \\ 
  & \le \left\| A+B \right\|+\left\| A-B \right\| \\ 
  & \le 4\omega \left( \left[ \begin{matrix}
    O & A  \\
    {{B}^{*}} & O  \\
 \end{matrix} \right] \right).
 \end{aligned}
 \end{equation}

Also noting \eqref{eq_w_average_w}, by the second inequality in \eqref{8}, we get the following  interesting inequalities
\[\begin{aligned}
   \frac{\left\| A+B \right\|+\left\| A-B \right\|}{2}&\le 2\omega \left( \left[ \begin{matrix}
   O & A  \\
   {{B}^{*}} & O  \\
\end{matrix} \right] \right) \\ 
 & \le \omega \left( A+{{B}^{*}} \right)+\omega \left( A-{{B}^{*}} \right).  
\end{aligned}\]
\end{remark}

The following result provides an integral version of \eqref{eq_fuad_mos}; where the numerical radius of convex combinations of operator matrices is used to refine the triangle inequality. Since its proof is similar to Theorem \ref{10}, we state it without details.
\begin{theorem}
Let $A,B\in \mathcal B\left( \mathcal H \right)$. Then
\[\left\| A+B \right\|\le 2\int\limits_{0}^{1}{\omega \left( \left[ \begin{matrix}
   O & \left( 1-t \right)A+tB  \\
   \left( 1-t \right){{B}^{*}}+t{{A}^{*}} & O  \\
\end{matrix} \right] \right)}dt\le \left\| A \right\|+\left\| B \right\|.\]
\end{theorem}

The matrix operator $\left[\begin{array}{cc}O&A\\B^*&O\end{array}\right]$ is further used to obtain the following improvement of \eqref{eq_abuamer}.
\begin{theorem}\label{11}
Let $A,B\in \mathcal B\left( \mathcal H \right)$.  Then for any $0 \le t \le 1$,
\[\begin{aligned}
  & \left\| A+B \right\|+\frac{\omega \left( \left[ \begin{matrix}
   O & A  \\
   {{B}^{*}} & O  \\
\end{matrix} \right] \right)-\omega \left( \left[ \begin{matrix}
   O & \left( 1-t \right)A+tB  \\
   \left( 1-t \right){{B}^{*}}+t{{A}^{*}} & O  \\
\end{matrix} \right] \right)}{R} \\ 
 & \le \max \left\{ \left\| A \right\|,\left\| B \right\| \right\}+\frac{1}{2}\left( \left\| {{\left| A \right|}^{\frac{1}{2}}}{{\left| B \right|}^{\frac{1}{2}}} \right\|+\left\| {{\left| {{B}^{*}} \right|}^{\frac{1}{2}}}{{\left| {{A}^{*}} \right|}^{\frac{1}{2}}} \right\| \right), 
\end{aligned}\]
where $R=\max \left\{ t,1-t \right\}$. In particular, if $A$ and $B$ are self-adjoint, we get
{\small
\[\left\| A+B \right\|+\frac{\omega \left( \left[ \begin{matrix}
   O & A  \\
   B & O  \\
\end{matrix} \right] \right)-\omega \left( \left[ \begin{matrix}
   O & \left( 1-t \right)A+tB  \\
   \left( 1-t \right)B+tA & O  \\
\end{matrix} \right] \right)}{R}\le \max \left\{ \left\| A \right\|,\left\| B \right\| \right\}+\left\| {{\left| A \right|}^{\frac{1}{2}}}{{\left| B \right|}^{\frac{1}{2}}} \right\|.\]
}
\end{theorem}
\begin{proof}

Combining \eqref{eq_need_prf} with the inequality \eqref{4}, we infer the desired result.
\end{proof}

\begin{remark}
It is worthwhile to mention here that if $A$ and $B$ are positive operators, then Theorem \ref{11} reduces to \cite{7}
\[\left\| A+B \right\|\le \max \left\{ \left\| A \right\|,\left\| B \right\| \right\}+\left\| {{A}^{\frac{1}{2}}}{{B}^{\frac{1}{2}}} \right\|.\]
This follows from the following point for positive operators \cite{8}
\begin{equation}\label{eq_ned_pf_remark}
\omega \left( \left[ \begin{matrix}
   O & \left( 1-t \right)A+tB  \\
   \left( 1-t \right)B+tA & O  \\
\end{matrix} \right] \right)=\omega \left( \left[ \begin{matrix}
   O & A  \\
   B & O  \\
\end{matrix} \right] \right)=\frac{1}{2}\left\| A+B \right\|.
\end{equation}

\end{remark}

Now we move to study inequalities for $\omega(AB)$, where $A,B\in\mathcal{B}(\mathcal{H})$. Interestingly, the following numerical radius inequality leads to a new proof of the arithmetic-geometric mean inequality for positive operators, as we shall see in Remark \ref{remark_amgm} below.
\begin{theorem}\label{12}
Let $A,B\in \mathcal B\left( \mathcal H \right)$.  Then for any $0 \le t \le 1$,
\[{{\omega }^{\frac{1}{2}}}\left( AB \right)\le \frac{1}{2}\left\| A+B^{*} \right\|+\frac{\omega \left( \left[ \begin{matrix}
   O & A  \\
   {{B}} & O  \\
\end{matrix} \right] \right)-\omega \left( \left[ \begin{matrix}
   O & \left( 1-t \right)A+tB^{*}  \\
   \left( 1-t \right){{B}}+t{{A}^{*}} & O  \\
\end{matrix} \right] \right)}{2r},\]
where $r=\min \left\{ t,1-t \right\}$.
\end{theorem}
\begin{proof}
By the second inequality in \eqref{6}, we have
\[2\omega \left( \left[ \begin{matrix}
   O & A  \\
   {{B}^{*}} & O  \\
\end{matrix} \right] \right)\le \left\| A+B \right\|+\frac{\omega \left( \left[ \begin{matrix}
   O & A  \\
   {{B}^{*}} & O  \\
\end{matrix} \right] \right)-\omega \left( \left[ \begin{matrix}
   O & \left( 1-t \right)A+tB  \\
   \left( 1-t \right){{B}^{*}}+t{{A}^{*}} & O  \\
\end{matrix} \right] \right)}{r}.\]
Thus,
\[\begin{aligned}
  & 2{{\omega }^{\frac{1}{2}}}\left( AB \right) \\ 
 & \le 2\max \left\{ {{\omega }^{\frac{1}{2}}}\left( AB \right),{{\omega }^{\frac{1}{2}}}\left( BA \right) \right\} \\ 
 & =2{{\omega }^{\frac{1}{2}}}\left( \left[ \begin{matrix}
   AB & O  \\
   O & BA  \\
\end{matrix} \right] \right) \\ 
 & =2{{\omega }^{\frac{1}{2}}}\left( {{\left[ \begin{matrix}
   O & A  \\
   B & O  \\
\end{matrix} \right]}^{2}} \right) \\ 
 & \le 2\omega \left( \left[ \begin{matrix}
   O & A  \\
   B & O  \\
\end{matrix} \right] \right)\quad({\text{by}}\;\eqref{eq_power_ineq}) \\ 
 & \le \left\| A+B^{*} \right\|+\frac{\omega \left( \left[ \begin{matrix}
   O & A  \\
   {{B}} & O  \\
\end{matrix} \right] \right)-\omega \left( \left[ \begin{matrix}
   O & \left( 1-t \right)A+tB^{*}  \\
   \left( 1-t \right){{B}}+t{{A}^{*}} & O  \\
\end{matrix} \right] \right)}{r}, 
\end{aligned}\]
which completes the proof.
\end{proof}

Now we use Theorem \ref{12} to prove the following arithmetic-geometric mean inequality for positive operators.
\begin{remark}\label{remark_amgm}
Let $A,B\in \mathcal B\left( \mathcal H \right)$ be two positive operators. It follows from Theorem \ref{12},
\[\begin{aligned}
\left\| {{A}^{\frac{1}{2}}}{{B}^{\frac{1}{2}}} \right\|&={{r}^{\frac{1}{2}}}\left( AB \right)\quad \text{(by \cite[(2.1)]{5})}\\
&\le {{\omega }^{\frac{1}{2}}}\left( AB \right)\\
&\le \frac{1}{2}\left\| A+B \right\|,
\end{aligned}\]
where \eqref{eq_ned_pf_remark} has been used together with the fact that $r(T)\le \omega(T)$ for any $T\in\mathcal{B}(\mathcal{H})$.
\end{remark}

While Theorem \ref{12} provides an upper bound of $\omega(AB)$ in terms of  $ \left[ \begin{matrix}
   O & A  \\
   {{B}} & O  \\
\end{matrix} \right]$, we have the following lower bound in terms of the same matrix operator.
\begin{theorem}\label{1}
Let $A,B\in \mathcal B\left( \mathcal H \right)$. Then
\[\omega \left( \left[ \begin{matrix}
   O & A  \\
   B & O  \\
\end{matrix} \right] \right)\le \sqrt{\max \left\{ \omega \left( AB \right),\omega \left( BA \right) \right\}+\underset{\lambda \in \mathbb{C}}{\mathop{\inf }}\,{{\left\| \left[ \begin{matrix}
   -\lambda I & A  \\
   B & -\lambda I  \\
\end{matrix} \right] \right\|}^{2}}},\]
where $I$ is the identity operator in $\mathcal{B}(\mathcal{H}).$
\end{theorem}
\begin{proof}
By the main result of \cite{9}, we can write
\[\begin{aligned}
   \max \left\{ \omega \left( A{{B}^{*}} \right),\omega \left( {{B}^{*}}A \right) \right\}&=\omega \left( \left[ \begin{matrix}
   A{{B}^{*}} & O  \\
   O & {{B}^{*}}A  \\
\end{matrix} \right] \right) \\ 
 & =\omega \left( \left[ \begin{matrix}
   O & A  \\
   {{B}^{*}} & O  \\
\end{matrix} \right]\left[ \begin{matrix}
   O & A  \\
   {{B}^{*}} & O  \\
\end{matrix} \right] \right) \\ 
 & =\omega \left( {{T}^{2}} \right) \\ 
 & \ge {{\omega }^{2}}\left( T \right)-\underset{\lambda \in \mathbb{C}}{\mathop{\inf }}\,{{\left\| T-\lambda I \right\|}^{2}} \\ 
 & ={{\omega }^{2}}\left( \left[ \begin{matrix}
   O & A  \\
   {{B}^{*}} & O  \\
\end{matrix} \right] \right)-\underset{\lambda \in \mathbb{C}}{\mathop{\inf }}\,{{\left\| \left[ \begin{matrix}
   -\lambda I & A  \\
   {{B}^{*}} & -\lambda I  \\
\end{matrix} \right] \right\|}^{2}},
\end{aligned}\]
which completes the proof.
\end{proof}

\begin{remark}
It follows from Theorem \ref{1} that for $X_i\in\mathcal{B}(\mathcal{H})$ $\left( i=1,2,3,4 \right)$,
\[\begin{aligned}
  & \omega \left( \left[ \begin{matrix}
   {{X}_{1}} & {{X}_{2}}  \\
   {{X}_{3}} & {{X}_{4}}  \\
\end{matrix} \right] \right) \\ 
 & =\omega \left( \left[ \begin{matrix}
   {{X}_{1}} & O  \\
   O & {{X}_{4}}  \\
\end{matrix} \right]+\left[ \begin{matrix}
   O & {{X}_{2}}  \\
   {{X}_{3}} & O  \\
\end{matrix} \right] \right) \\ 
 & \le \omega \left( \left[ \begin{matrix}
   {{X}_{1}} & O  \\
   O & {{X}_{4}}  \\
\end{matrix} \right] \right)+\omega \left( \left[ \begin{matrix}
   O & {{X}_{2}}  \\
   {{X}_{3}} & O  \\
\end{matrix} \right] \right) \\ 
 & \le \max \left\{ \omega \left( {{X}_{1}} \right),\omega \left( {{X}_{4}} \right) \right\}+\sqrt{\max \left\{ \omega \left( {{X}_{2}}{{X}_{3}} \right),\omega \left( {{X}_{3}}{{X}_{2}} \right) \right\}+\underset{\lambda \in \mathbb{C}}{\mathop{\inf }}\,{{\left\| \left[ \begin{matrix}
   -\lambda I & {{X}_{2}}  \\
   {{X}_{3}} & -\lambda I  \\
\end{matrix} \right] \right\|}^{2}}}.  
\end{aligned}\]
\end{remark}
\begin{remark}
Notice that
\[\begin{aligned}
   r\left( {{X}_{1}}{{X}_{2}}+{{X}_{3}}{{X}_{4}} \right) & =r\left( \left[ \begin{matrix}
   {{X}_{1}}{{X}_{2}}+{{X}_{3}}{{X}_{4}} & O  \\
   O & O  \\
\end{matrix} \right] \right) \\ 
 & =r\left( \left[ \begin{matrix}
   {{X}_{1}} & {{X}_{3}}  \\
   O & O  \\
\end{matrix} \right]\left[ \begin{matrix}
   {{X}_{2}} & O  \\
   {{X}_{4}} & O  \\
\end{matrix} \right] \right) \\ 
 & =r\left( \left[ \begin{matrix}
   {{X}_{2}} & O  \\
   {{X}_{4}} & O  \\
\end{matrix} \right]\left[ \begin{matrix}
   {{X}_{1}} & {{X}_{3}}  \\
   O & O  \\
\end{matrix} \right] \right) \\ 
 & =r\left( \left[ \begin{matrix}
   {{X}_{2}}{{X}_{1}} & {{X}_{2}}{{X}_{3}}  \\
   {{X}_{4}}{{X}_{1}} & {{X}_{4}}{{X}_{3}}  \\
\end{matrix} \right] \right) \\ 
 & \le \omega \left( \left[ \begin{matrix}
   {{X}_{2}}{{X}_{1}} & {{X}_{2}}{{X}_{3}}  \\
   {{X}_{4}}{{X}_{1}} & {{X}_{4}}{{X}_{3}}  \\
\end{matrix} \right] \right). 
\end{aligned}\]
If in the above inequality we put ${{X}_{1}}={{e}^{\textup i\theta }}A$, ${{X}_{2}}=B$, ${{X}_{3}}={{e}^{-\textup i\theta }}{{B}^{*}}$, and ${{X}_{4}}={{A}^{*}}$, we reach
\begin{equation}\label{13}
\left\| {{\operatorname{\mathfrak Re}}^{\textup i\theta }}AB \right\|\le \frac{1}{2}\omega \left( \left[ \begin{matrix}
   {{e}^{\textup i\theta }}BA & {{e}^{-\textup i\theta }}B{{B}^{*}}  \\
   {{e}^{\textup i\theta }}{{A}^{*}}A & {{\left( {{e}^{\textup i\theta }}BA \right)}^{*}}  \\
\end{matrix} \right] \right).
\end{equation}
This indicates the relation between the numerical radius of the product of two operators and the numerical radius of $2\times 2$ operator matrices.

The case $A=U{{\left| T \right|}^{1-t}}$ and $B={{\left| T \right|}^{t}}$, in \eqref{13}, implies
\[\left\| \mathfrak R{{e}^{\textup i\theta }}T \right\|\le \frac{1}{2}\omega \left( \left[ \begin{matrix}
   {{e}^{\textup i\theta }}\widetilde{{{T}_{t}}} & {{e}^{-\textup i\theta }}{{\left| T \right|}^{2t}}  \\
   {{e}^{\textup i\theta }}{{\left| T \right|}^{2\left( 1-t \right)}} & {{\left( {{e}^{\textup i\theta }}\widetilde{{{T}_{t}}} \right)}^{*}}  \\
\end{matrix} \right] \right),\quad 0\le t \le 1,\]
where $\widetilde{{{T}_{t}}}$ is the weighted Aluthge transform of $T$ defined by $\widetilde{{{T}_{t}}}=|T|^tU|T|^{1-t},$ where $U$ is the partial isometry appearing in the polar decomposition in $T=U|T|.$ 

Notice that, if we replace $A=\sqrt{\frac{\left\| B \right\|}{\left\| A \right\|}}A$ and $B=\sqrt{\frac{\left\| A \right\|}{\left\| B \right\|}}B$, in \eqref{13}, we also have
\begin{equation}\label{14}
\left\| {{\operatorname{\mathfrak Re}}^{\textup i\theta }}AB \right\|\le \frac{1}{2}\omega \left( \left[ \begin{matrix}
   {{e}^{\textup i\theta }}BA & {{e}^{-\textup i\theta }}\frac{\left\| A \right\|}{\left\| B \right\|}B{{B}^{*}}  \\
   {{e}^{\textup i\theta }}\frac{\left\| B \right\|}{\left\| A \right\|}{{A}^{*}}A & {{\left( {{e}^{\textup i\theta }}BA \right)}^{*}}  \\
\end{matrix} \right] \right),
\end{equation}
and
\[\left\| {{\operatorname{\mathfrak Re}}^{\textup i\theta }}T \right\|\le \frac{1}{2}\omega \left( \left[ \begin{matrix}
   {{e}^{\textup i\theta }}\widetilde{{{T}_{t}}} & {{e}^{-\textup i\theta }}{{\left\| T \right\|}^{1-2t}}{{\left| T \right|}^{2t}}  \\
   {{e}^{\textup i\theta }}{{\left\| T \right\|}^{2t-1}}{{\left| T \right|}^{2\left( 1-t \right)}} & {{\left( {{e}^{\textup i\theta }}\widetilde{{{T}_{t}}} \right)}^{*}}  \\
\end{matrix} \right] \right).\]
\end{remark}
To better understand how the above relations help obtain the numerical radius of the product of two operators, we give an example. Recall that in \cite[Corollary 2]{8}, Abu-Omar and Kittaneh proved that if $\mathcal H_1$ and $\mathcal H_2$ are Hilbert spaces and $\mathbb X=\left[ \begin{matrix}
   {{X}_{1}} & {{X}_{2}}  \\
   {{X}_{3}} & {{X}_{4}}  \\
\end{matrix} \right]$  is an operator matrix with
$X_1\in \mathcal B(\mathcal H_1)$, $X_2\in \mathcal B(\mathcal H_2,\mathcal H_1)$, $X_3\in \mathcal B(\mathcal H_1,\mathcal H_2)$, and $X_4\in \mathcal B(\mathcal H_2)$, then
\begin{equation*}
\omega \left( \mathbb X \right)\le \frac{1}{2}\left( \omega \left( {{X}_{1}} \right)+\omega \left( {{X}_{4}} \right)+\sqrt{{{\left( \omega \left( {{X}_{1}} \right)-\omega \left( {{X}_{4}} \right) \right)}^{2}}+4{{\omega }^{2}}\left( \mathbb E \right)} \right),
\end{equation*}
where $\mathbb E=\left[ \begin{matrix}
   O & {{X}_{2}}  \\
   {{X}_{3}} & O  \\
\end{matrix} \right]$.
In the same paper (see \cite[Remark 6]{8}), it has been shown that
\[{{\omega }}\left( \mathbb E \right)\le \min \left\{ {{\alpha }_{1}},{{\alpha }_{2}} \right\}\]
where
\[{{\alpha }_{1}}=\frac{1}{4}\sqrt{\left\| {{\left| {{X}_{2}} \right|}^{2}}+{{\left| X_{3}^{*} \right|}^{2}} \right\|+2\omega \left( {{X}_{3}}{{X}_{2}} \right)}\;\text{ and }\;{{\alpha }_{2}}=\frac{1}{4}\sqrt{\left\| {{\left| X_{2}^{*} \right|}^{2}}+{{\left| {{X}_{3}} \right|}^{2}} \right\|+2\omega \left( {{X}_{2}}{{X}_{3}} \right)}.\]
Combining these two inequalities we get 
\[\omega \left( \left[ \begin{matrix}
   {{X}_{1}} & {{X}_{2}}  \\
   {{X}_{3}} & {{X}_{4}}  \\
\end{matrix} \right] \right)\le \frac{1}{2}\left( \omega \left( {{X}_{1}} \right)+\omega \left( {{X}_{4}} \right)+\sqrt{{{\left( \omega \left( {{X}_{1}} \right)-\omega \left( {{X}_{4}} \right) \right)}^{2}}+4\min \left\{ \alpha _{1}^{2},\alpha _{2}^{2} \right\}} \right).\]
Now, using this and \eqref{14}, we have
\[\left\| {{\operatorname{\mathfrak Re}}^{\textup i\theta }}AB \right\|\le \frac{1}{2}\left( \omega \left( BA \right)+\min \left\{ {{\beta }_{1}},{{\beta }_{2}} \right\} \right),\]
where
\[{{\beta }_{1}}=\frac{1}{2}\sqrt{\left\| \frac{{{\left\| A \right\|}^{2}}}{{{\left\| B \right\|}^{2}}}{{\left| {{B}^{*}} \right|}^{4}}+\frac{{{\left\| B \right\|}^{2}}}{{{\left\| A \right\|}^{2}}}{{\left| A \right|}^{4}} \right\|+2\omega \left( {{\left| A \right|}^{2}}{{\left| {{B}^{*}} \right|}^{2}} \right)},\]
and
\[{{\beta }_{2}}=\frac{1}{2}\sqrt{\left\| \frac{{{\left\| A \right\|}^{2}}}{{{\left\| B \right\|}^{2}}}{{\left| {{B}^{*}} \right|}^{4}}+\frac{{{\left\| B \right\|}^{2}}}{{{\left\| A \right\|}^{2}}}{{\left| A \right|}^{4}} \right\|+2\omega \left( {{\left| {{B}^{*}} \right|}^{2}}{{\left| A \right|}^{2}} \right)}.\]
This implies,
\[\omega \left( AB \right)\le \frac{1}{2}\omega \left( BA \right)+\frac{1}{4}\sqrt{\left\| \frac{{{\left\| A \right\|}^{2}}}{{{\left\| B \right\|}^{2}}}{{\left| {{B}^{*}} \right|}^{4}}+\frac{{{\left\| B \right\|}^{2}}}{{{\left\| A \right\|}^{2}}}{{\left| A \right|}^{4}} \right\|+2\min \left\{ \omega \left( {{\left| A \right|}^{2}}{{\left| {{B}^{*}} \right|}^{2}} \right),\omega \left( {{\left| {{B}^{*}} \right|}^{2}}{{\left| A \right|}^{2}} \right) \right\}}.\]
We also have by \eqref{13},
\[\omega \left( AB \right)\le \frac{1}{2}\omega \left( BA \right)+\frac{1}{4}\sqrt{\left\| {{\left| {{B}^{*}} \right|}^{4}}+{{\left| A \right|}^{4}} \right\|+2\min \left\{ \omega \left( {{\left| A \right|}^{2}}{{\left| {{B}^{*}} \right|}^{2}} \right),\omega \left( {{\left| {{B}^{*}} \right|}^{2}}{{\left| A \right|}^{2}} \right) \right\}}.\]

\vskip 0.3 true cm 

\noindent{\tiny (M. Sababheh) Vice president, Princess Sumaya University for Technology, Amman, Jordan}
	
\noindent	{\tiny\textit{E-mail address:} sababheh@psut.edu.jo; sababheh@yahoo.com}

\vskip 0.3 true cm

\noindent{\tiny (C. Conde)  Instituto de Ciencias, Universidad Nacional de General Sarmiento  and  Consejo Nacional de Investigaciones Cient\'ificas y Tecnicas, Argentina}

\noindent{\tiny \textit{E-mail address:} cconde@campus.ungs.edu.ar}

\vskip 0.3 true cm

\noindent{\tiny (H. R. Moradi) Department of Mathematics, Payame Noor University (PNU), P.O. Box, 19395-4697, Tehran, Iran
	
\noindent	\textit{E-mail address:} hrmoradi@mshdiau.ac.ir}
\end{document}